\documentclass[a4paper,11pt]{amsart}

\usepackage{amsfonts}
\usepackage{amssymb}
\usepackage{amsmath}
\usepackage{amsthm}

\newtheorem{theorem}{Theorem}[section]
\newtheorem{proposition}{Proposition}[section]
\newtheorem{lemma}{Lemma}[section]

\theoremstyle{definition}

\numberwithin{equation}{section}
\numberwithin{theorem}{section}
\numberwithin{lemma}{section}
\numberwithin{corollary}{section}
\numberwithin{proposition}{section}

\newcommand{\wC}{\widetilde{C}}
\newcommand{\wD}{\widetilde{D}}
\newcommand{\wP}{\widetilde{P}}
\newcommand{\wV}{\widetilde{V}}

\newcommand{\ds}{\displaystyle}

\allowdisplaybreaks[1]

\newcommand{\N}{\mathbb{N}}

\begin{document}

\title[On a New Modification of Baskakov Operators]
      {On a New Modification of Baskakov Operators \\ with Higher Order of Approximation}

\author[I. Gadjev]{Ivan Gadjev$^1$}
\address{$^1$\,Faculty of Mathematics and Informatics, Sofia University St. Kliment Ohridski, \newline
         \indent 5, J. Bourchier Blvd, 1164 Sofia, Bulgaria}
\email{gadjev@fmi.uni-sofia.bg}

\author[P. Parvanov]{Parvan Parvanov$^2$}
\address{$^2$\,Faculty of Mathematics and Informatics, Sofia University St. Kliment Ohridski, \newline
         \indent 5, J. Bourchier Blvd, 1164 Sofia, Bulgaria}
\email{pparvan@fmi.uni-sofia.bg}

\author[R. Uluchev]{Rumen Uluchev$^3$}
\address{$^3$\,Faculty of Mathematics and Informatics, Sofia University St. Kliment Ohridski, \newline
         \indent 5, J. Bourchier Blvd, 1164 Sofia, Bulgaria}
\email{rumenu@fmi.uni-sofia.bg}

\thanks{$^{1,2,3}$\,This study is financed by the European Union-NextGenerationEU, through the National
        Recovery and Resilience Plan of the Republic of Bulgaria, project No BG-RRP-2.004-0008.}



\date{}

\begin{abstract}
  A new Goodman-Sharma modification of the Baskakov operator is presented for approximation of
  bounded and continuous on $[0,\,\infty)$ functions. In our study on the approximation error of
  the proposed operator we prove direct and strong converse theorems with respect to a related
  K-functional. This operator is linear but not positive. However it has the advantage of a higher
  order of approximation compared to the Goodman-Sharma variant of the Baskakov operator defined
  in 2005 by Finta.
\end{abstract}

\maketitle


\bigskip
\section{Introduction}

In 1957, Baskakov \cite{Ba1957} suggested the linear positive operators
\begin{equation} \label{eq:1.1}
  B_n(f,x) = \sum_{k=0}^{\infty} \,f\Big(\frac{k}{n}\Big)P_{n,k}(x)
\end{equation}
for approximation of bounded and continuous on $[0,\,\infty)$ functions $f$, where
$$
  P_{n,k}(x) = \binom{n+k-1}{k}x^k(1+x)^{-n-k}, \qquad k=0,1,\ldots,
$$
can be considered as Baskakov basis functions.

Following the Goodman and Sharma modification \cite{GoSh1988} of the Bernstein polynomials,
Finta~\cite{Fi2005} introduced a variant of the operators $B_n$ for functions $f$ which are Lebesgue
measurable on $(0,\infty)$ with a finite limit $f(0)$ as $x\to 0$:
\begin{equation} \label{eq:1.2}
  \begin{gathered}
    V_n(f,x) = \sum _{k=0}^{\infty } \,v_{n,k}(f)P_{n,k}(x), \\
    v_{n,0}(f) = f(0), \qquad v_{n,k}(f) = (n+1)\int_0^{\infty} P_{n+2,k-1}(t)f(t)\,dt, \ \ k\in\N,
  \end{gathered}
\end{equation}
Finta proved a strong converse result of Type B (in the terminology of \cite{DiIv1993}) for $V_n$.
The research on the operators \eqref{eq:1.2} were continued in \cite{Fi2006, GuAg2007, GuNoBeGu2006}.

Later Ivanov and Parvanov \cite{IvPa2011} investigated the uniform weighted approximation error
of the Baskakov-type operators $V_n$ for weights of the form
$\big(\frac{x}{1+x}\big)^{\gamma_0}(1+x)^{\gamma_{\infty}}$, $\gamma_0,\gamma_{\infty}\in[-1,0]$,
by establishing direct and strong converse theorems in terms of the weighted K-functional.

Recently Jabbar and Hassan \cite{JaHa2024} studied a family of Baskakov-type operators, where the
Baskakov basis functions $P_{n,k}$ in $B_n f$ are replaced by linear combinations of Baskakov
basis functions of lower degree with coefficients being polynomials of appropriate degree.
The benefit is obtaining a better order of approximation than the classical Baskakov operator.
Jabar and Hassan proved certain direct results on the approximation rate of these operators and
attached results of numerical experiments.

The ideas presented in \cite{JaHa2024} suggested the authors of the current paper to define and
explore new operators with higher order of approximation.

As usual, $D f(x)=\frac{df(x)}{dx}=f'(x)$, $D^2 f(x)=\frac{d^2 f(x)}{dx^2}=f''(x)$, and we denote
by
$$
  \psi(x) = x(1+x)
$$
the weight function which is naturally related to the second order differential operator for the
Baskakov type operators. Also, we set
\begin{equation} \label{eq:1.3}
  \wD f(x) := \psi(x)f''(x),
\end{equation}
and
$$
  \wD^2 f := \wD\wD f, \qquad \wD^{r+1} f := \wD \wD^r f, \qquad r=2,3\ldots\,.
$$

Our study is on the operators explicitly defined by
\begin{equation} \label{eq:1.4}
  \wV_n(f,x) = \sum_{k=0}^{\infty} v_{n,k}(f)\wP_{n,k}(x), \qquad x\in[0,\infty),
\end{equation}
with basis functions
\begin{equation} \label{eq:1.5}
    \wP_{n,k}(x) = P_{n,k}(x) - \frac{1}{n}\,\wD P_{n,k}(x)
\end{equation}

The operators $\wV_n$ relate to operators $V_n$ in the same manner as some operators in
\cite{JaHa2024} relate to the classical Baskakov operators \eqref{eq:1.1}.

By $C[0,\infty)$ we denote the space of all continuous on $[0,\infty)$ functions, and
$L_{\infty}[0,\infty)$ stands for the space of all Lebesgue measurable and essentially bounded on
$[0,\infty)$ functions equipped with the uniform norm $\|\cdot\|$.
Let us set
$$
  W^2(\psi) = \big\{g\,:\,g,g'\in AC_{loc}(0,\infty),\ \wD g \in L_\infty[0,\infty)\big\},
$$
where $AC_{loc}(0,\infty)$ consists of the functions which are absolutely continuous on $[a,b]$ for
every $[a,b]\subset (0,\infty)$.

By $W^2_0(\psi)$ we denote the subspace of $W^2(\psi)$ of functions $g$ satisfying the additional
boundary condition
$$
  \lim_{x\to 0^+} \wD g(x) = 0.
$$

For functions $f\in C[0,\infty)$ and $t>0$ we define the K-functional
\begin{equation} \label{eq:1.6}
  K(f,t) = \inf \big\{ \| f-g \| + t \| \wD ^2g\| \,:\, g\in W^2_0(\psi),\ \wD g\in W^2(\psi)\big\}.
\end{equation}

Below we investigate the error rate for functions $f\in C[0,\infty)$ approximated by the
Goodman-Sharma modification of the Baskakov operator \eqref{eq:1.4}. Direct and strong converse
theorems are proved in means of the above K-functional $K(f,t)$ and we summarize our main results
in the following statements.

\begin{theorem} \label{th:1.1}
  Let $\,n\ge 2$, $n\in\mathbb{N}$. Then for all $f\in C[0,\infty)$, there holds
  $$
    \big\|\wV_n f-f\big\| \le 3K\Big(f,\frac{1}{n^2}\Big).
  $$
\end{theorem}

\begin{theorem} \label{th:1.2}
  There exist constants $C,L>0$ such that for all $f\in C[0,\infty)$ and $\ell,n\in \mathbb{N}$
  with  $\ell\ge Ln$ there holds
  $$
    K\Big(f,\frac{1}{n^2}\Big) \le
    C\,\frac{\ell^2}{n^2}\,\big(\big\|\wV_n f - f \big\| + \big\|\wV_{\ell} f - f \big\| \big).
  $$
  In particular,
  $$
    K\Big(f,\frac{1}{n^2}\Big) \le
    C\,\big(\big\|\wV_n f - f \big\| + \big\|\wV_{Ln} f - f \big\| \big).
  $$
  The constants $C$ are independent of the function $f$, $\ell$ and $n$.
\end{theorem}

The article is organized in the following way. Section~1 is an introduction to the topic. We~give
notations, define a new modification of the Baskakov operator and claim our main results. Section~2
includes preliminary and auxiliary statements. In Section~3 we present an estimation of the norm
of the operator $\wV_n$, a Jackson type inequality and a proof of the direct theorem. The converse
result for the modified Baskakov operator \eqref{eq:1.4} is discussed in Section~4. Inequalities of
the Voronovskaya type and Bernstein type for $\wV_n$ are proved using the differential
operator~$\wD$, defined in \eqref{eq:1.3}. Theorem~\ref{th:1.2} represents a strong converse
inequality of Type~B in the Ditzian-Ivanov classification~\cite{DiIv1993}. Finally, a proof of the
converse theorem is given.


\bigskip
\section{Preliminaries and Auxiliary Results}

The central moments of the Baskakov operator are defined by
$$
  \mu_{n,j}(x) = B_n\big((t-x)^j,x\big) = \sum_{k=0}^{\infty} \Big(\frac{k}{n}-x\Big)^j P_{n,k}(x),
  \qquad j=0,1,\ldots\,.
$$
We set $P_{n,k}:=0$ if $k<0$. The next proposition summarizes several well known relations and
formulae for the Baskakov basis polynomials.

\begin{proposition}[see, e.g. {\cite[pp.~38--39]{IvPa2011}}] \label{pr:2.1}
  (a) The following identities are valid:
      \begin{align} \label{eq:2.1}
        & k P_{n,k}(x) = nx P_{n+1,k-1}(x), \qquad (n+k)P_{n,k}(x) = n(1+x)P_{n+1,k}(x), \smallskip\\
        & k(k-1)P_{n,k}(x) = n(n+1)x^2 P_{n+2,k-2}(x), \label{eq:2.2} \smallskip\\
        & (n+k)(n+k+1)P_{n,k}(x) = n(n+1)(1+x)^2 P_{n+2,k}(x), \label{eq:2.3} \smallskip\\
        & \tfrac{x}{1+x}\,P_{n,k}(x) = \tfrac{k+1}{n+k}\,P_{n,k+1}(x), \label{eq:2.4} \smallskip\\
        & \tfrac{1+x}{x}\,P_{n,k}(x) = \tfrac{n+k-1}{k}\,P_{n,k-1}(x), \label{eq:2.5} \smallskip\\
        & P_{n,k}'(x) = n\big[P_{n+1,k-1}(x) - P_{n+1,k}(x)\big], \label{eq:2.6} \smallskip\\
        & P_{n,k}''(x) = n(n+1)\big[P_{n+2,k-2}(x)-2P_{n+2,k-1}(x)+P_{n+2,k}(x)\big]. \label{eq:2.7}
      \end{align}

  (b) For the moments $\mu_{n,j}(x)$, $j=0,1,...,4$, we have:
      \begin{align*}
        \mu_{n,0}(x) & = B_n\big((t-x)^0,x\big) = 1, \\
        \mu_{n,1}(x) & = B_n\big((t-x),x\big) = 0, \\
        \mu_{n,2}(x) & = B_n\big((t-x)^2,x\big) = \frac{\psi(x)}{n}, \\
        \mu_{n,3}(x) & = B_n\big((t-x)^3,x\big) = \frac{(1+2x)\psi(x)}{n^2}, \\
        \mu_{n,4}(x) & = B_n\big((t-x)^4,x\big) = \frac{3(n+2)\psi^2(x)}{n^3} + \frac{\psi(x)}{n^3}.
      \end{align*}
\end{proposition}

The operators $V_n$ and $\wV_n$ together with the differential operator $\wD$ have specific
commutative relations collected in the statement below. We have also added some other important
properties.

\begin{proposition} \label{pr:2.2}
  If the operators $V_n$, $\wV_n$ and the differential operator $\wD$ are defined as in
  \eqref{eq:1.2}, \eqref{eq:1.4} and \eqref{eq:1.3}, respectively, then

  \medskip
  (a) \ $\wD V_n f = V_n \wD f$ for $f\in W^2_0(\psi)$;

  \smallskip
  (b) \ $\wV_n f = V_n\big(f-\frac{1}{n}\,\wD f\big)$ for $f\in W^2_0(\psi)$;

  \smallskip
  (c) \ $\wD \wV_n f = \wV_n \wD f$ for $f\in W^2_0(\psi)$;

  \smallskip
  (d) \ $V_n \wV_n f = \wV_n V_n f$ for $f\in W^2_0(\psi)$;

  \smallskip
  (e) \ $\wV_m \wV_n f = \wV_n \wV_m f$ for $f\in W^2_0(\psi)$;

  \smallskip
  (f) \ $\lim\limits_{n\to \infty} \wV_n f = f$ for $f\in W^2(\psi)$;

  \smallskip
  (g) \ $\big\| \wD V_n f\big\| \le \big\|\wD f\big\|$  for  $f\in W^2(\psi)$.
\end{proposition}

\begin{proof}
For the proof of (a), see \cite[Theorem 2.5]{IvPa2011}.

We have
\begin{align*}
  \wV_{n} f & = \sum_{k=0}^{\infty} v_{n,k}(f)\wP_{n,k} \\
            & = v_{n,0}(f)\Big(P_{n,0} - \frac{1}{n}\,\wD P_{n,0}\Big)
                + \sum_{k=1}^{\infty} v_{n,k}(f)\Big(P_{n,k} - \frac{1}{n}\,\wD P_{n,k}\Big) \\
                 & = v_{n,0}(f)P_{n,0} + \sum_{k=1}^{\infty} v_{n,k}(f) P_{n,k} -\frac{\psi}{n}\bigg(v_{n,0}(f)D^2P_{n,0} + \sum_{k=1}^{\infty} v_{n,k}(f)D^2P_{n,k}\bigg) \\
            & = V_n f - \frac{1}{n}\,\psi D^2V_n f=V_n f - \frac{1}{n}\,\wD V_n f.
\end{align*}
Then from (a) we obtain
$$
  \wV_n f = V_n f - \frac{1}{n}\,\wD V_n f = V_n f - \frac{1}{n}\,V_n \wD f
          = V_n\Big(f-\frac{1}{n}\,\wD f\Big),
$$
which proves (b).

Now, commutative properties (c) and (d) follow from (b) and (a):
$$
  \wD \wV_n f = \wD V_n\Big(f-\frac{1}{n}\,\wD f\Big) = V_n\Big(\wD f - \frac{1}{n}\,\wD \wD f\Big)
               = \wV_n(\wD f),
$$
and
\begin{align*}
  V_n \wV_n f & = V_n\Big(V_nf-\frac{1}{n}\,\wD V_n f \Big) = V_n V_n f - \frac{1}{n}\,V_n \wD V_n f \\
              & = V_n V_n f - \frac{1}{n}\,\wD V_n V_n f = V_n (V_n f)-\frac{1}{n}\,\wD V_n (V_n f)
                 = \wV_n V_n f.
\end{align*}

The operators $\wV_n$ commute in the sense of (e), since
\begin{align*}
  \wV_m \wV_n f & = \wV_m \Big(V_n f - \frac{1}{n}\,\wD V_n f\Big) \\
                & = V_m \Big(V_n f - \frac{1}{n}\,\wD V_n f\Big)
                    -\frac{1}{m} \wD V_m\Big(V_n f -\frac{1}{n}\,\wD V_n f\Big) \\
                & = V_m V_n \Big(f - \frac{m+n}{mn}\,\wD f + \frac{1}{mn}\,\wD^2 f\Big).
\end{align*}
The same expression in the last line we obtain for $\wV_n \wV_m f$ because of properties
(a), (b) and $V_m V_n f = V_n V_m f$.

In \cite[Lemma 3.2]{IvPa2011} it was proved that
\begin{align} \label{eq:2.8}
  \|V_n f  - f\| & \le \frac{1}{n}\,\big\|\wD f\big\|, \\
  \big\|V_n \wD f\big\| & \le \big\|\wD f\big\|. \notag
\end{align}
Hence
$$
  \|\wV_n f - f\| = \Big\|V_n f  -\frac{1}{n}\,V_n \wD f - f\Big\|
                   \le \|V_n f - f\| + \frac{1}{n}\,\big\|V_n \wD f\big\|
                   \le \frac{2}{n}\,\|\wD f\|.
$$
Therefore $\lim\limits_{n\to \infty} \|\wV_n f - f\|=0$, i.e. property (f) holds true.

For a proof of the inequality in (g) we refer to \cite[Lemma 2.8]{IvPa2011} with $w=1$.
\end{proof}

\smallskip
Now we introduce a function that will prove useful in our further investigations:
\begin{align} \label{eq:2.9}
  T_{n,k}(x) & := k(k-1)\frac{1+x}{x} - 2k(n+k) + (n+k)(n+k+1)\frac{x}{1+x} \\
             & = n\bigg[-1-\frac{1+2x}{\psi(x)}\Big(\frac{k}{n}-x\Big) + \frac{n}{\psi(x)}\Big(\frac{k}{n}-x\Big)^2\bigg].
               \notag
\end{align}

Observe that
\begin{align}
  T_{n,k}'(x)  & = - \frac{k(k-1)}{x^2} + \frac{(n+k)(n+k+1)}{(1+x)^2}, \label{eq:2.10} \\
  T_{n,k}''(x) & = \frac{2k(k-1)}{x^3} - \frac{2(n+k)(n+k+1)}{(1+x)^3}. \label{eq:2.11}
\end{align}

\begin{proposition} \label{pr:2.3}
(a) The following relation concerning functions $P_{n,k}$, $T_{n,k}$ and differential operator
$\wD$ are valid:
\begin{equation} \label{eq:2.12}
  \wD P_{n,k}(x) = T_{n,k}(x) P_{n,k}(x).
\end{equation}

(b) If $\,\alpha$ is an arbitrary real number, then
$$
  \Phi(\alpha) := \sum_{k=0}^{\infty} \Big(\alpha - \frac{1}{n}\,T_{n,k}(x)\!\Big)^2 P_{n,k}(x)
               = \alpha^2 + 2 + \frac{2}{n}\,.
$$
\end{proposition}

\begin{proof}
(a) \ It is easy to see that
\begin{equation} \label{eq:2.13}
  \psi(x)P_{n+2,k-1}(x) = \frac{k(n+k)}{n(n+1)}\,P_{n,k}(x).
\end{equation}
By using \eqref{eq:2.7} and then \eqref{eq:2.2}, \eqref{eq:2.13}, \eqref{eq:2.3} we obtain
\begin{align*}
  \psi(x)P_{n,k}''(x)
    & = n(n+1)\big[\psi(x)P_{n+2,k-2}(x)-2\psi(x)P_{n+2,k-1}(x)+\psi(x)P_{n+2,k}(x)\big] \\
    & = n(n+1)\bigg[\psi(x)\,\frac{k(k-1)}{n(n+1)x^2}\,P_{n,k}(x)-2\,\frac{k(n+k)}{n(n+1)}\,P_{n,k}(x) \\
    & \hspace*{21mm}  + \psi(x)\frac{(n+k)(n+k+1)}{n(n+1)(1+x)^2}\,P_{n,k}(x)\bigg] \\
    & = \bigg[k(k-1)\frac{1+x}{x} - 2k(n+k) + (n+k)(n+k+1)\frac{x}{1+x}\bigg]P_{n,k}(x) \\
    & = T_{n,k}(x)P_{n,k}(x),
\end{align*}
i.e. the identity \eqref{eq:2.12}.

(b) \ We apply the formulae for the Baskakov operator moments in Proposition~\ref{pr:2.1}\,(b):
\begin{align*}
  \Phi(\alpha)
    & = \sum_{k=0}^{\infty} \bigg[\alpha + 1 + \frac{1+2x}{\psi(x)} \Big(\frac{k}{n}-x\Big)
                          - \frac{n}{\psi(x)}\Big(\frac{k}{n}-x\Big)^{\!2}\bigg]^2 P_{n,k}(x) \\
    & = \sum_{k=0}^{\infty} \bigg[(\alpha+1)^2 \!+ \frac{(1+2x)^2}{\psi^2(x)}\Big(\frac{k}{n}-x\Big)^2
                            \!+ \frac{n^2}{\psi^2(x)}\Big(\frac{k}{n}-x\Big)^4
                            \!+ \frac{2(\alpha+1)(1+2x)}{\psi(x)}\Big(\frac{k}{n}-x\Big) \\
    & \hspace*{15mm}        - \frac{2(\alpha+1)n}{\psi(x)}\Big(\frac{k}{n}-x\Big)^2
                            \!- \frac{2n(1+2x)}{\psi^2(x)}\Big(\frac{k}{n}-x\Big)^3 \bigg] P_{n,k}(x) \\
    & = (\alpha+1)^2\mu_{n,0}(x) + \frac{(1+2x)^2}{\psi^2(x)}\,\mu_{n,2}(x) + \frac{n^2}{\psi^2(x)}\,\mu_{n,4}(x)
                            + \frac{2(\alpha+1)(1+2x)}{\psi(x)}\,\mu_{n,1}(x) \\
    & \qquad                - \frac{2(\alpha+1)n}{\psi(x)}\,\mu_{n,2}(x) - \frac{2n(1+2x)}{\psi^2(x)}\,\mu_{n,3}(x) \\
    & = (\alpha+1)^2\cdot 1 + \frac{(1+2x)^2}{\psi^2(x)}\,\frac{\psi(x)}{n}
                            + \frac{n^2}{\psi^2(x)}\,\frac{(3n+6)\psi^2(x)+\psi(x)}{n^3} \\
    & \qquad                + \frac{2(\alpha+1)(1+2x)}{\psi(x)}\cdot 0 - \frac{2(\alpha+1)n}{\psi(x)}\,\frac{\psi(x)}{n}
                            - \frac{2n(1+2x)}{\psi^2(x)}\,\frac{(1+2x)\psi(x)}{n^2} \\
    & = (\alpha+1)^2 + \frac{1+4\psi(x)}{n\psi(x)} + \frac{(3n+6)\psi(x)+1}{n\psi(x)} - 2(\alpha+1)
                     - \frac{2(1+4\psi(x))}{n\psi(x)} \\
    & = \alpha^2 + 2\alpha + 1 + \frac{1}{n\psi(x)} + \frac{4}{n} + 3 +\frac{6}{n} + \frac{1}{n\psi(x)}
                               -2\alpha - 2 - \frac{2}{n\psi(x)} - \frac{8}{n} \\
   & = \alpha^2 + 2 +\frac{2}{n}.
\end{align*}
\end{proof}

\begin{proposition} \label{pr:2.4}
The following relations hold true:


\medskip
(a) \ $T_{n+1,k-1}(x)P_{n+1,k-1}(x) + T_{n+1,k}(x)P_{n+1,k}(x)$

    \qquad\quad $ = (k-2)(n+k-1)P_{n+1,k-2}(x) - (k-1)(n+k)P_{n+1,k-1}(x)$

    \qquad\qquad $- k(n+k+1)P_{n+1,k}(x) + (k+1)(n+k+2)P_{n+1,k+1}(x)$;


\smallskip
(b) \ $- \dfrac{\psi(x)}{n}\,T_{n,k}'(x)\,P_{n,k}'(x)
       = k(n+k-1)P_{n+1,k-2}(x) - (k-1)(n+k)P_{n+1,k-1}(x)$

      \hspace*{50mm} $- k(n+k+1)P_{n+1,k}(x) + (k+1)(n+k)P_{n+1,k+1}(x)$;


\smallskip
(c) \ $\dfrac{\psi(x)}{n}\,T_{n,k}''(x)P_{n,k}(x) = 2(n+k-1)P_{n+1,k-2}(x) - 2(k+1)P_{n+1,k+1}(x)$.
\end{proposition}

\begin{proof}
(a) \ By \eqref{eq:2.5} and \eqref{eq:2.4} we have
$$
  T_{n,k}(x) P_{n,k}(x) = (k-1)(n+k-1)\,P_{n,k-1}(x) - 2k(n+k)\,P_{n,k}(x)
                          + (k+1)(n+k+1)\,P_{n,k+1}(x).
$$
Then (a) follows immediately:
\begin{align*}
  & T_{n+1,k-1}(x)P_{n+1,k-1}(x) + T_{n+1,k}(x)P_{n+1,k}(x) \\
  & \qquad = (k-2)(n+k-1)P_{n+1,k-2}(x) - 2(k-1)(n+k)P_{n+1,k-1}(x) + k(n+k+1)P_{n+1,k}(x) \\
  & \qquad\quad + (k-1)(n+k)P_{n+1,k-1}(x) - 2k(n+k+1)P_{n+1,k}(x) + (k+1)(n+k+2)P_{n+1,k+1}(x) \\
  & \qquad = (k-2)(n+k-1)P_{n+1,k-2}(x) - (k-1)(n+k)P_{n+1,k-1}(x) \\
  & \qquad\quad - k(n+k+1)P_{n+1,k}(x) + (k+1)(n+k+2)P_{n+1,k+1}(x).
\end{align*}

\noindent
(b) \ From \eqref{eq:2.10}, \eqref{eq:2.6}, \eqref{eq:2.4} and \eqref{eq:2.5}
we have
\begin{align*}
  & - \frac{\psi(x)}{n}\,T_{n,k}'(x)\,P_{n,k}'(x) \\
  & \qquad = \frac{\psi(x)}{n}\bigg[\frac{k(k-1)}{x^2} - \frac{(n+k)(n+k+1)}{(1+x)^2}\bigg]n\big(P_{n+1,k-1}(x)-P_{n+1,k}(x)\big) \\
  & \qquad = k(k-1)\frac{1+x}{x}P_{n+1,k-1}(x) - k(k-1)\frac{1+x}{x}P_{n+1,k}(x) \\
  & \qquad\quad -(n+k)(n+k+1)\frac{x}{1+x}P_{n+1,k-1}(x) + (n+k)(n+k+1)\frac{x}{1+x}P_{n+1,k}(x) \\
  & \qquad = k(n+k-1)P_{n+1,k-2}(x) - (k-1)(n+k)P_{n+1,k-1}(x) \\
  & \qquad\quad - k(n+k+1)P_{n+1,k}(x) + (k+1)(n+k)P_{n+1,k+1}(x).
\end{align*}

\noindent
(c) \ From \eqref{eq:2.11}, \eqref{eq:2.4}, \eqref{eq:2.5} and \eqref{eq:2.1}
we have
\begin{align*}
  \frac{\psi(x)}{n}\,T_{n,k}''(x)P_{n,k}(x)
  & = \frac{\psi(x)}{n}\bigg[\frac{2k(k-1)}{x^3} - \frac{2(n+k)(n+k+1)}{(1+x)^3}\bigg]P_{n,k}(x) \\
  & = \frac{2}{n}\bigg[\frac{k(k-1)}{x}\frac{1+x}{x}\,P_{n,k}(x) - \frac{(n+k)(n+k+1)}{1+x}\frac{x}{1+x}\,P_{n,k}(x)\bigg] \\
  & = \frac{2}{n}\bigg[\frac{k(k-1)}{x}\frac{n+k-1}{k}P_{n,k-1}(x) - \frac{(n+k)(n+k+1)}{1+x}\frac{k+1}{n+k}P_{n,k+1}(x)\bigg] \\
  & = 2(n+k-1)\frac{(k-1)P_{n,k-1}(x)}{nx} - 2(k+1)\frac{(n+k+1)P_{n,k+1}(x)}{n(1+x)} \\
  & = 2(n+k-1)P_{n+1,k-2}(x) - 2(k+1)P_{n+1,k+1}(x).
\end{align*}
\end{proof}

\begin{proposition} \label{pr:2.5}
  Let
  $$
    \lambda(n) := \sum_{k=n}^{\infty} \frac{1}{k(k+1)^2}\,, \qquad
    \theta(n) := \sum_{k=n}^{\infty} \frac{1}{k^2(k+1)^2}\,, \qquad n\in \mathbb{N}.
  $$
  Then, for $n\ge 2$ the next inequalities are satisfied:
  \begin{gather} \label{eq:2.14}
    \frac{1}{3n^2} \le \lambda(n) \le \frac{1}{n^2}\,, \\
    \theta(n) \le  \frac{4}{9n^3}\,. \label{eq:2.15}
  \end{gather}
\end{proposition}

\begin{proof}
For $\lambda(n)$ we have
$$
  \frac{1}{3n^2}\le \frac1{2n(n+1)} = \sum_{k=n}^{\infty} \frac{1}{k(k+1)(k+2)} < \lambda(n) <
  \sum_{k=n}^{\infty} \frac{1}{(k-1)k(k+1)} = \frac1{2n(n-1)} \le \frac{1}{n^2}\,.
$$

For the upper estimate of $\theta(n)$ we obtain
$$
  \theta(n) < \sum_{k=n}^{\infty} \frac{1}{(k-1)k(k+1)(k+2)} = \frac{1}{3n(n^2-1)} \le \frac{4}{9n^3}\,.
$$
\end{proof}


\bigskip
\section{A Direct Theorem}

First we estimate from above the norm of the operator $\wV_n$ defined in \eqref{eq:1.4}.

\begin{lemma} \label{le:3.1}
  If $\,n\ge 2$, $n\in \mathbb{N}$ and $f\in C[0,\infty)$, then
  \begin{equation} \label{eq:3.1}
    \big\|\wV_n f\big\| \le 2\,\|f\|.
  \end{equation}
\end{lemma}

\begin{proof}
From \eqref{eq:1.5} and Proposition~\ref{pr:2.3}\,(a) we have
$$
  \wP_{n,k}(x) = P_{n,k}(x) - \frac{1}{n}\,\wD P_{n,k}(x) = \Big(1-\frac{1}{n}\,T_{n,k}(x)\Big)P_{n,k}(x).
$$
Taking into account that $|v_{n,k}(f)|\le \|f\| $ we get for $x\in[0,\infty )$,
\begin{align*}
  \big|\wV_n(f,x)\big|
  & = \left| \sum_{k=0}^\infty v_{n,k}(f)\wP_{n,k}(x) \right|
    \le \sum_{k=0}^\infty  |v_{n,k}(f)|\,\big|\wP_{n,k}(x)\big| \\
  & \le \|f\| \sum_{k=0}^\infty  \big|\wP_{n,k}(x)\big|
    = \|f\| \sum_{k=0}^\infty  \Big|1-\frac{1}{n}\,T_{n,k}(x)\Big| P_{n,k}(x).
\end{align*}
Then, the Cauchy inequality yields
$$
  \big|\wV_n(f,x)\big| \le
  \|f\| \sqrt{\sum_{k=0}^\infty  \Big(1-\frac{1}{n}\,T_{n,k}(x)\Big)^2 P_{n,k}(x)}\,\sqrt{\sum_{k=0}^\infty  P_{n,k}(x)}.
$$
Since $\sum_{k=0}^\infty  P_{n,k}(x)=1$ identically, by Proposition~\ref{pr:2.3}\,(b) with $\alpha=1$ we find
$$
  \big|\wV_n(f,x)\big| \le \sqrt{3+\frac{2}{n}}\,\|f\| < 2\,\|f\|, \qquad x\in [0,\infty).
$$
Hence, inequality \eqref{eq:3.1} follows.
\end{proof}

\smallskip
Now we prove a Jackson type inequality.

\begin{lemma} \label{le:3.2}
  If $\,n\ge 2$, $n\in \mathbb{N}$, $f\in W^2_0(\psi)$ and $\wD f\in W^2(\psi)$, then
  $$
    \big\|\wV_n f - f\big\| \le \frac{1}{n^2}\|\wD^2 f\|.
  $$
\end{lemma}

\begin{proof}
According to \cite[Lemma 2.2]{IvPa2011} we have
$$
  V_k f - V_{k+1} f = \frac{1}{k(k+1)}\,\wD V_k f.
$$
Combining the latter with Proposition~\ref{pr:2.2}\,(a)--(b) we obtain
\begin{align*}
  \wV_k f - \wV_{k+1} f
    & = V_k f - \frac{1}{k}\,\wD V_k f - V_{k+1} f + \frac{1}{k+1}\,\wD V_{k+1} f \\
    & = \Big(\frac{1}{k}-\frac{1}{k+1}\Big)\wD V_k f + \frac{1}{k+1}\,\wD V_{k+1} f - \frac{1}{k}\,\wD V_k f \\
    & = - \frac{1}{k+1}\big(\wD V_k f - \wD V_{k+1} f\big) \\
    & = - \frac{1}{k+1}\big(V_k \wD f - V_{k+1} \wD f\big) \\
    & = - \frac{1}{k+1}\cdot\frac{1}{k(k+1)}\, \wD V_k \wD f,
\end{align*}
i.e.
$$
  \wV_k f - \wV_{k+1} f = - \frac{1}{k(k+1)^2}\,\wD V_k \wD f.
$$
Therefore for every $s>n$ we have
\begin{gather*}
 \wV_n f - \wV_s f = \sum_{k=n}^{s-1} \big(\wV_k f - \wV_{k+1} f\big)
                   = - \sum_{k=n}^{s-1} \frac{1}{k(k+1)^2}\,\wD V_k \wD f.
\end{gather*}
Letting $s\rightarrow\infty$ by Proposition~\ref{pr:2.2}\,(a) and (f) we obtain
\begin{equation} \label{eq:3.2}
  \wV_n f - f = - \sum_{k=n}^{\infty} \frac{1}{k(k+1)^2}\,\wD V_k \wD f
              = - \sum_{k=n}^{\infty} \frac{1}{k(k+1)^2}\,V_k \wD^2 f.
\end{equation}
Then Proposition~\ref{pr:2.2}\,(g) yields
$$
  \|\wV_n f - f\| \le \sum_{k=n}^{\infty} \frac{1}{k(k+1)^2}\,\big\|\wD V_k \wD f\big\|
                  \le \sum_{k=n}^{\infty} \frac{1}{k(k+1)^2}\,\big\|\wD^2 f\big\|
$$
and from \eqref{eq:2.14} we conclude
$$
  \big\|\wV_n f - f\big\| \le \frac{1}{n^2}\,\big\|\wD^2 f\big\|.
$$
\end{proof}

Based on both lemmas above we prove a direct result for the approximation rate of functions
$f\in C[0,\infty)$ by the operators \eqref{eq:1.4} in means of the K-functional defined in
\eqref{eq:1.6}.

\smallskip
\begin{proof}[Proof of Theorem~\ref{th:1.1}]
Let $g$ be arbitrary function such that $g\in W^2_0(\psi)$ and $\wD g\in W^2(\psi)$. Then by
Lemma~\ref{le:3.1} and Lemma~\ref{le:3.2} we have
\begin{align*}
  \big\|\wV_n f-f\big\|
    & \le \big\|\wV_n f-\wV_n g\big\| + \big\|\wV_n g-g\big\| + \|g-f\| \\
    & \le 3\|f-g\| + \frac{1}{n^2}\,\big\|\wD^2 g\big\| \\
    & \le 3\Big(\|f-g\|+\frac{1}{n^2}\,\big\|\wD^2 g\big\|\Big).
\end{align*}
Taking infimum over all functions $g\in W^2_0(\psi)$ with $\wD g\in W^2(\psi)$ we obtain
$$
  \big\|\wV_n f-f\big\| \le 3\,K\Big(f,\frac{1}{n}\Big).
$$
\end{proof}


\bigskip
\section{A Strong Converse Inequality}

First, we will prove a Voronovskaya type result for the operator $\wV_n$.

\begin{lemma} \label{le:4.1}
  If $\,n\ge 2$, $n\in \mathbb{N}$, $\ds \,\lambda(n) = \sum_{k=n}^{\infty} \frac{1}{k(k+1)^2}$,
  $\ds \,\theta(n) = \sum_{k=n}^{\infty} \frac{1}{k^2(k+1)^2}$, and $f\in C[0,\infty)$ is such that
  $f\in W^2_0(\psi)$, $\wD^2 f\in W^2(\psi)$, then
  $$
    \big\|\wV_n f - f + \lambda(n)\wD^2 f\big\| \le \theta(n)\,\big\|\wD^3 f\big\|.
  $$
\end{lemma}

\begin{proof}
We have from \eqref{eq:3.2} that
$$
  \wV_n f - f + \lambda(n)\wD^2 f
    = - \sum_{k=n}^{\infty} \frac{V_k\wD^2 f}{k(k+1)^2} + \sum_{k=n}^{\infty} \frac{\wD^2 f}{k(k+1)^2}
    = \sum_{k=n}^{\infty} \frac{\wD^2 f - V_k \wD^2 f}{k(k+1)^2}.
$$
Then
$$
  \big\|\wV_n f - f + \lambda(n)\wD^2 f\big\|
  \le \sum_{k=n}^{\infty} \frac{1}{k(k+1)^2}\,\big\|\wD^2 f - V_k \wD^2 f\big\|.
$$
Using \eqref{eq:2.8} with $\wD^2 f$ replacing $f$ we obtain
$$
  \big\|\wV_n f - f + \lambda(n)\wD^2 f\big\|
    \le \sum_{k=n}^{\infty} \frac{1}{k(k+1)^2}\cdot\frac{1}{k}\,\big\|\wD \wD^2 f\big\|
    = \theta(n)\,\big\|\wD^3 f\big\|.
$$
\end{proof}

We need the next inequality of Bernstein type.

\begin{lemma} \label{le:4.2}
  Let $n\in\mathbb{N}$ and $f\in C[0,\infty )$. Then for $n\ge 17$ the inequality
  $$
    \|\wD \wV_n f\| \le \wC\, n\|f\|
  $$
  holds true, where $\wC=6+4\sqrt{3}$\,.
\end{lemma}

\begin{proof}
Since
$$
  \big|\wD\wV_n(f,x)\big| \le \sum_{k=0}^\infty |v_{n,k}(f)|\,\big|\wD\wP_{n,k}(x)\big|
                              \le \|f\| \sum_{k=0}^\infty \big|\wD\wP_{n,k}(x)\big|,
$$
it is sufficient to prove that
$$
  \sum_{k=0}^\infty \big|\wD \wP_{n,k}(x)\big| = \sum_{k=0}^\infty \big|\psi(x)\wP_{n,k}''(x)\big|
  \le (6+4\sqrt{3})\,n, \qquad n\ge 17.
$$

Remind that by \eqref{eq:2.12} we have
$$
  \wD P_{n,k}(x) = T_{n,k}(x)P_{n,k}(x),
$$
hence
\begin{align*}
  \wP_{n,k}(x)   & = P_{n,k}(x) - \frac{1}{n}\,\wD P_{n,k}(x)
                   = \Big(1 - \frac{1}{n}\,T_{n,k}(x)\!\Big)P_{n,k}(x), \\
  \wP_{n,k}''(x) & = - \frac{1}{n}\,T_{n,k}''(x) P_{n,k}(x) - \frac{2}{n}\,T_{n,k}'(x) P_{n,k}'(x)
                     + \Big(1 - \frac{1}{n}\,T_{n,k}(x)\!\Big) P_{n,k}''(x).
\end{align*}
Then,
\begin{align*}
  \wD \wP_{n,k}(x)
  & = - \frac{\psi(x)}{n}\,T_{n,k}''(x) P_{n,k}(x) - \frac{2\psi(x)}{n}\,T_{n,k}'(x) P_{n,k}'(x)
                       + \Big(1 - \frac{1}{n}\,T_{n,k}(x)\!\Big) \psi(x) P_{n,k}''(x) \\
  & = - \frac{\psi(x)}{n}\,T_{n,k}''(x) P_{n,k}(x) - \frac{2\psi(x)}{n}\,T_{n,k}'(x) P_{n,k}'(x)
                       + \Big(1 - \frac{1}{n}\,T_{n,k}(x)\!\Big) T_{n,k}(x) P_{n,k}(x).
\end{align*}

By using Proposition~\ref{pr:2.4}\,(a)--(c) we obtain
$$
  - \frac{\psi(x)}{n}\,T_{n,k}'(x) P_{n,k}'(x) =
  T_{n+1,k-1}(x)P_{n+1,k-1}(x) + T_{n+1,k}(x)P_{n+1,k}(x) + \frac{\psi(x)}{n}T_{n,k}''(x) P_{n,k}(x)
$$
and then
\begin{align*}
  \wD \wP_{n,k}(x) & = \frac{\psi(x)}{n}\,T_{n,k}''(x) P_{n,k}(x)
                       + 2\big[T_{n+1,k-1}(x)P_{n+1,k-1}(x) +T_{n+1,k}(x)P_{n+1,k}(x)\big] \\
                   & \quad + \Big(1 - \frac{1}{n}\,T_{n,k}(x)\!\Big) T_{n,k}(x) P_{n,k}(x).
\end{align*}

Therefore
\begin{align*}
  \sum_{k=0}^{\infty}\big|\wD \wP_{n,k}(x)\big| & \le a_n(x) + b_n(x) + c_n(x), \\
  \intertext{where}
  a_n(x) & = \frac{\psi(x)}{n} \sum_{k=0}^{\infty} \big|T_{n,k}''(x)\big| P_{n,k}(x), \\
  b_n(x) & = 2 \sum_{k=0}^{\infty} |T_{n+1,k-1}(x) | P_{n+1,k-1}(x) + 2 \sum_{k=0}^{\infty} |T_{n+1,k}(x)| P_{n+1,k}(x), \\
  c_n(x) & = \sum_{k=0}^{\infty} \Big|\Big(1 - \frac{1}{n}\,T_{n,k}(x)\!\Big) T_{n,k}(x)\Big| P_{n,k}(x).
\end{align*}

\smallskip
{\sl (i) \ Estimation of $a_n(x)$.}
We will estimate $a_n(x)$ separately in the intervals $\big(0,\frac1n\big]$ and
$\big[\frac1n,\infty\big)$.

Let $x\in\big(0,\frac1n\big]$. Then, from \eqref{eq:2.11}, $T_{n,0}''(x)<0$ and $T_{n,1}''(x)<0$.
Since $\big(\frac{1+x}{x}\big)^3$ is a decreasing function for $x\in \big(0,\frac1n\big]$ and
$\frac{(n+k)(n+k+1)}{k(k-1)}$ decreases for $k\ge 2$, we have
$$
  \Big(\frac{1+x}{x}\Big)^3 > (n+1)^3 > \frac{(n+1)(n+2)}{2} \ge \frac{(n+k)(n+k+1)}{k(k-1)}.
$$
Hence
$$
  T_{n,k}''(x) = 2\Big(\frac{k(k-1)}{x^3} - \frac{(n+k)(n+k+1)}{(1+x)^3}\Big) > 0, \qquad k\ge 2.
$$
Applying the latter inequalities, Proposition~\ref{pr:2.4}\,(c) and \eqref{eq:2.1} we obtain
\begin{align*}
  a_n(x) & = \frac{\psi(x)}{n} \sum_{k=0}^{\infty} |T_{n,k}''(x)| P_{n,k}(x) \\
         & = \frac{\psi(x)}{n} \bigg[\!\frac{n(n+1)}{(1+x)^3}\frac{1}{(1+x)^n}+\frac{(n+1)(n+2)}{(1+x)^3}\frac{nx}{(1+x)^{n+1}}+ \sum_{k=2}^{\infty} T_{n,k}''(x) P_{n,k}(x)\!\bigg] \\
         & = 2\frac{\psi(x)}{n}\bigg[\frac{n(n+1)}{(1+x)^{n+3}}+\frac{n(n+1)(n+2)x}{(1+x)^{n+4}}\bigg]
             + \frac{\psi(x)}{n} \sum_{k=0}^{\infty} T_{n,k}(x) P_{n,k}(x)\\
         & = 2(n+1)\,\frac{x}{(1+x)^{n+2}} + 2n(n+1)(n+2)\frac{x^2}{(1+x)^{n+3}} \\
         & \quad + 2\sum_{k=0}^{\infty} \big[(n+k-1)P_{n+1,k-2}(x)-(k+1)P_{n+1,k+1}(x)\big] \\
         & = 2(n+1)\,\frac{x}{(1+x)^{n+2}} + 2n(n+1)(n+2)\,\frac{x^2}{(1+x)^{n+3}} \\
         & \quad + 2 \sum_{k=0}^{\infty} (n+1)(1+x)P_{n+2,k-1}(x) - 2\sum_{k=0}^{\infty} (n+1)xP_{n+2,k}(x) \\
         & = 2(n+1)\frac{x}{(1+x)^{n+2}} + 2n(n+1)(n+2)\,\frac{x^2}{(1+x)^{n+3}} + 2(n+1)(1+x) - 2(n+1)x \\
         & < \frac{2(n+1)}{n} + \frac{2(n+1)(n+2)}{n^2} + 2(n+1).
\end{align*}

Therefore,
\begin{equation} \label{eq:4.1}
  a_n(x) < 3n, \qquad\text{for all} \ \ n\ge 8, \ \ x\in\big[0,\tfrac1n\big).
\end{equation}

Now, let $x\in\big[\frac1n,\infty\big)$. By the Cauchy inequality we have
\begin{align*}
  a_n(x) & = \frac{2 \psi(x)}{n} \sum_{k=0}^{\infty} \Big|\frac{k(k-1)}{x^3}-\frac{(n+k)(n+k+1)}{(1+x)^3}\Big| P_{n,k}(x) \\
         & \le \frac{2\psi(x)}{n} \sqrt{\sum_{k=0}^{\infty} \Big(\frac{k(k-1)}{x^3}-\frac{(n+k)(n+k+1)}{(1+x)^3}\Big)^{\!2} P_{n,k}(x)} \,\sqrt{\sum_{k=0}^{\infty} P_{n,k}(x)}
\end{align*}
and then
\begin{equation} \label{eq:4.2}
  a_n^2(x) = \frac{4\psi^2(x)}{n^2} \sum_{k=0}^{\infty} \Big(\frac{k(k-1)}{x^3}-\frac{(n+k)(n+k+1)}{(1+x)^3}\Big)^{\!2} P_{n,k}(x).
\end{equation}

From \eqref{eq:2.1} and \eqref{eq:2.2} we get
\begin{align} \label{eq:4.3}
  & \sum_{k=0}^{\infty} k^2(k-1)^2 P_{n,k}(x) = n(n+1)(n+2)\Big((n+3)x^4+4x^3+\frac{2x^2}{n+2}\Big), \\
  & \sum_{k=0}^{\infty} k(k-1)(n+k)(n+k+1)P_{n,k}(x) = n(n+1)(n+2)(n+3)x^2(1+x)^2 \label{eq:4.4}, \\
  & \sum_{k=0}^{\infty} (n+k)^2(n+k+1)^2P_{n,k}(x) = n(n+1)(n+2)(1+x)^2 \label{eq:4.5} \\[-5pt]
  & \hspace*{60mm} \times \Big(\!(n+3)(1+x)^2-4(1+x)+\frac{2}{n+2}\Big). \notag
\end{align}

After straightforward computations \eqref{eq:4.2}--\eqref{eq:4.5} yields
$$
  a^2_n(x) = \frac{4(n+1)(n+2)}{n}\bigg[n+3+4\bigg(\frac{(1+x)^2}{x}-\frac{x^2}{1+x}\bigg)+\frac2{n+2}\bigg(\frac{(1+x)^2}{x^2}-\frac{x^2}{(1+x)^2}\bigg)\bigg].
$$
Observe that functions $g_1(x)=\frac{(1+x)^2}{x}-\frac{x^2}{1+x}$ and
$g_2(x)=\frac{(1+x)^2}{x^2}-\frac{x^2}{(1+x)^2}$ are decreasing for $x\in (0,\infty)$. Hence, for
$x\in \big[\frac1n,\infty)$,
$$
  g_1(x) \le g_1\Big(\frac1n\Big) = \frac{(n+1)^2}{n} - \frac{1}{n(n+1)}, \qquad
  g_2(x) \le g_2\Big(\frac1n\Big) = (n+1)^2 + \frac1{(n+1)^2}.
$$
Then
\begin{align*}
  a^2_n(x) & \le \frac{4(n+1)(n+2)}{n}\bigg[n+3+\frac{4(n+1)^2}{n}+\frac{4(n+1)^2}{n+2}+\frac{2}{(n+1)^2(n+2)}-\frac4{n(n+1)}\bigg] \\
           & < \frac{4(n+1)(n+2)}{n}\bigg[n+3+\frac{4(n+1)^2}{n}+\frac{4(n+1)^2}{n+2}\bigg] \\
           & <36n^2,
\end{align*}
i.e.
\begin{equation} \label{eq:4.6}
  a_n(x) < 6n, \qquad\text{for all} \ \ n\ge 17, \ \ x\in\big[\tfrac1n,\infty\big).
\end{equation}

Now from \eqref{eq:4.1} and \eqref{eq:4.6} we get
\begin{equation} \label{eq:4.7}
  a_n(x) \le 6\,n, \qquad\text{for all} \ \ n\ge 17, \ \ x\in [0,\infty).
\end{equation}

\smallskip
{\sl (ii) \ Estimation of $b_n(x)$.}
\ By the Cauchy inequality and Proposition~\ref{pr:2.3}\,(b) with $\alpha=0$ we obtain
$$
  \sum_{k=0}^{\infty} \left|T_{n,k}(x)\right| P_{n,k}(x)
    \le \sqrt{\sum_{k=0}^{\infty} T_{n,k}^2(x) P_{n,k}(x)}\, \sqrt{\sum_{k=0}^{\infty} P_{n,k}(x)}
    = \sqrt{\Phi(0)n^2} = n\sqrt{2+\frac{2}{n}} < \sqrt{3}\,n,
$$
and hence for all $n\ge 2$,
\begin{equation} \label{eq:4.8}
  b_n(x) \le 2\sqrt{3}\,n, \qquad x\in [0,\infty).
\end{equation}

\smallskip
{\sl (iii) \ Estimation of $c_n(x)$.} \ Similarly, by applying the Cauchy inequality and
Proposition~\ref{pr:2.3}\,(b) with $\alpha=0$ and $\alpha=1$:
\begin{align*}
  c_n(x) & = \sum_{k=0}^{\infty} \Big|T_{n,k}(x)\Big(1 - \frac{1}{n}\,T_{n,k}(x)\Big)\Big| P_{n,k}(x) \\
         & \le \sqrt{\sum_{k=0}^{\infty} T_{n,k}^2(x) P_{n,k}(x)}\,
               \sqrt{\sum_{k=0}^{\infty} \Big(1-\frac{1}{n}\,T_{n,k}(x)\Big)^2 P_{n,k}(x)} \\
         & = \sqrt{\Phi(0)\,n^2}\,\sqrt{\Phi(1)} = n\sqrt{2+\frac{2}{n}}\,\sqrt{3+\frac{2}{n}}\,.
\end{align*}
Then for all $n\ge 2$,
\begin{equation} \label{eq:4.9}
  c_n(x) \le 2\sqrt{3}\,n, \qquad x\in [0,\infty).
\end{equation}

From \eqref{eq:4.7}--\eqref{eq:4.9} we obtain that for all $n\ge 17$,
$$
  \sum_{k=0}^n \big|\wD \wP_{n,k}(x)\big| \le a_n(x) + b_n(x) + c_n(x) \le \wC\,n, \qquad
  x\in [0,\infty),
$$
i.e.
$$
  \big\|\wD\wV_n f\big\| \le \wC n\|f\|, \qquad \wC = 6 + 4\sqrt{3}\,.
$$
\end{proof}

\smallskip
Now we are ready to prove a strong converse inequality of Type B according to Ditzian-Ivanov
classification, following their approach in~\cite{DiIv1993}.

\smallskip
\begin{proof}[Proof of Theorem~\ref{th:1.2}]
Let $n\in\mathbb{N}$, $n\ge 2$, $f\in C[0,\infty)$ and $\lambda(n)$, $\theta(n)$ be defined as in
Proposition~\ref{pr:2.5}. We apply the Voronovskaya type inequality in Lemma~\ref{le:4.1} for the
operator $\wV_{\ell}$ and $f$ replaced with $\wV_n^3 f$. Then
\begin{align*}
  \lambda(\ell) \big\|\wD^2 \wV_n^3 f\big\|
    & = \big\|\lambda(\ell)\wD^2 \wV_n^3 f\big\| \\
    & = \big\|\wV_{\ell} \wV_n^3 f - \wV_n^3 f +\lambda(\ell)\wD^2 \wV_n^3 f
        - \wV_{\ell} \wV_n^3 f + \wV_n^3 f\big\| \\
    & \le \big\|\wV_{\ell} \wV_n^3 f - \wV_n^3 f +\lambda(\ell)\wD^2 \wV_n^3 f\big\|
          + \big\|\wV_{\ell} \wV_n^3 f - \wV_n^3 f\big\| \\
    & \le \theta(\ell) \big\|\wD^3 \wV_n^3 f\big\| + \big\|\wV_n^3\big(\wV_{\ell} f - f\big)\big\|.
\end{align*}
Using Lemma~\ref{le:4.2} for the function $\wD^2 \wV_n^2 f$ and successively three times
Lemma~\ref{le:3.1} we get
\begin{align*}
  \lambda(\ell) \big\|\wD^2 \wV_n^3 f\big\|
    & \le \wC\,n\,\theta(\ell) \big\|\wD^2 \wV_n^2 f\big\| + 8\big\|\wV_{\ell}f - f\big\| \\
    & = \wC\,n\,\theta(\ell) \big\|\wD^2 \wV_n^2 (f - \wV_n f) + \wD^2 \wV_n^3 f\big\|
        + 8\big\|\wV_{\ell}f - f\big\| \\
    & \le \wC\,n\,\theta(\ell) \big\|\wD^2 \wV_n^2 (f - \wV_n f)\big\|
          + \wC\,n\,\theta(\ell) \big\|\wD^2 \wV_n^3 f\big\| + 8\big\|\wV_{\ell}f - f\big\|.
\end{align*}
Now, application of the Bernstein type inequality Lemma~\ref{le:4.2} twice for $f-\wV_n f$ yields
$$
  \lambda(\ell) \big\|\wD^2 \wV_n^3 f\big\|
  \le \wC^3 n^3\theta(\ell) \big\|f - \wV_n f\big\|
      + 8\,\big\|\wV_{\ell} - f\big\| + \wC\,n\,\theta(\ell) \big\|\wD^2 \wV_n^3 f\big\|.
$$
From inequalities \eqref{eq:2.14} and \eqref{eq:2.15} of Proposition~\ref{pr:2.5} we obtain
$$
  \frac{1}{2\ell^2} \big\|\wD^2 \wV_n^3 f\big\|
  \le \frac{4\wC^3 n^3}{9\ell^3} \big\|f - \wV_n f\big\|
      + 8\big\|\wV_{\ell} - f\big\| + \frac{4\wC n}{9\ell^3} \big\|\wD^2 \wV_n^3 f\big\|.
$$
Let us choose $\ell$ sufficiently large such that
$$
  \frac{4\wC n}{9\ell^3} \le \frac{1}{4\ell^2}, \qquad\text{i.e.}\qquad \ell \ge \frac{16\wC}{9}\,n.
$$
If we set $L=\frac{16\wC}{9}$, for all integers $\ell\ge Ln$ we have
\begin{align}
  \frac{1}{2\ell^2} \big\|\wD^2 \wV_n^3 f\big\|
    & \le \frac{4\wC^3 n^3}{9\ell^3} \big\|f - \wV_n f\big\|
        + 8\big\|\wV_{\ell} - f\big\| + \frac{1}{4\ell^2} \big\|\wD^2 \wV_n^3 f\big\|, \notag\\
  \frac{1}{4\ell^2} \big\|\wD^2 \wV_n^3 f\big\|
    & \le \frac{4\wC^3 n^3}{9\ell^3} \big\|f - \wV_n f\big\| + 8\big\|\wV_{\ell} - f\big\|, \notag\\
  \frac{1}{n^2} \big\|\wD^2 \wV_n^3 f\big\|
    & \le \wC^2\,\big\|f - \wV_n f\big\|
        + 32\frac{\ell^2}{n^2}\,\big\|\wV_{\ell} - f\big\|. \label{eq:4.10}
\end{align}
By Lemma~\ref{le:3.1},
\begin{align} \label{eq:4.11}
  \big\|f-\wV_n^3 f\big\|
   \le \big\|f - \wV_n f\big\| + \big\|\wV_n f - \wV_n^2 f\big\| + \big\|\wV_n^2 f - \wV_n^3 f\big\|
   \le 7\big\|f - \wV_n f\big\|.
\end{align}
Since $\wV_n^3 f\in W^2_0(\psi)$, from \eqref{eq:4.10} and \eqref{eq:4.11} it follows
\begin{align*}
  K\Big(f,\frac{1}{n^2}\Big)
    & = \inf \Big\{\|f-g\| + \frac{1}{n^2}\,\big\|\wD^2 g\big\|: \ g\in W^2_0(\psi),\,\wD g\in W^2(\psi)  \Big\} \\
    & \le \big\|f-\wV_n^3 f\| + \frac{1}{n^2}\,\big\|\wD^2 \wV_n^3 f\big\| \\
    & \le \big(7+\wC^2\big)\big\|\wV_n f-f\big\|
          + 32\,\frac{\ell^2}{n^2}\,\big\|\wV_{\ell} f-f\big\|.
\end{align*}

Hence, we obtain the following upper estimate of the K-functional,
$$
  K\Big(f,\frac{1}{n^2}\Big) \le
  C\,\frac{\ell^2}{n^2}\big(\big\|\wV_n f - f \big\| + \big\|\wV_{\ell} f - f) \big\| \big),
$$
for all $\ell\ge Ln$, where $C=7+\wC^2$ and $L=\frac{16\wC}{9}$, $\wC=6+4\sqrt{3}$\,.
\end{proof}



\end{document}